\documentclass[a4paper,12pt]{article}
\usepackage[top=2.5cm,bottom=2.5cm,left=2.5cm,right=2.5cm]{geometry}
\usepackage{cite,amsmath,amssymb}
\usepackage[margin=1cm,%
font=small,%
format=hang,%
labelsep=period,%
labelfont=bf]{caption}
\usepackage{epsf,epsfig,enumerate, amsfonts,amsgen,amstext,amsbsy,amsopn,amsthm}
\usepackage{ebezier,eepic}
\usepackage{color}
\usepackage{multirow}
\newtheorem{thm}{Theorem}[section]

\newtheorem{lem}{Lemma}[section]
\newtheorem{false statement}{False statement}

\newtheorem{fact}{Fact}

\theoremstyle{definition}

\newtheorem{prob}{Problem}

\newcommand{\bL}{\b{\textit{L}}}
\newcommand{\bN}{\b{\textit{N}}}

\begin{document}
\pagestyle{empty}

\vspace*{2cm}

\begin{center}
 {\LARGE\bf Wiener Index, Harary Index and Hamiltonicity
of Graphs}\\[6mm]

{\bf Hongbo Hua$^a$}, {\bf Bo Ning$^b$}\footnote{Corresponding author. Email:
bo.ning@tju.edu.cn (B. Ning)} \\[6mm]
{\small $^a$Faculty of Mathematics and Physics,
Huaiyin Institute of Technology \\ Huai'an, Jiangsu, 223003, China\\[4mm]
$^b$Center for Applied Mathematics,
Tianjin University\\ Tianjin, 300072, P. R. China.} \\[6mm]

(Received  September 18, 2016)\\[8mm]
\end{center}

\begin{abstract}
In this paper, we prove tight sufficient conditions for traceability
and Hamiltonicity of connected graphs with given minimum degree, in
terms of Wiener index and Harary index. We also prove some result on
Hamiltonicity of balanced bipartite graphs in the similar fashion.
In two recent papers \cite{LDJ2016,LDJ2017}, Liu et al.
corrected some previous work on traceability of connected graphs in
terms of Wiener index and Harary index, respectively, such as
\cite{HW2013,Y2013}. We generalize these results and give short
and unified proofs. All results in this paper are best possible.
\end{abstract}

 \normalsize
\baselineskip=0.30in

\section{Introduction}
Let $G$ be a graph. For two vertices $u,v$ of $G$, the \emph{distance}
between $u$ and $v$ in $G$, denoted by $d_G(u,v)$, is the length of a
shortest path from $u$ to $v$ in $G$. We denote by $diam(G)$ the diameter
of $G$, and denote by $\delta(G)$ the minimum degree of $G$. For two
graphs $G$ and $H$, we denote the \emph{union} of $G$ and $H$ by $G+H$,
and the \emph{join} of $G$ and $H$ by $G\vee H$. A graph is called
\emph{Hamiltonian} (\emph{traceable}) if there is a cycle (path) including
all vertices in it. A bipartite graph is called \emph{balanced} if its
each partition set has the same number of vertices. For terminology and
notation not defined here, we refer the reader to West \cite{W1996}.

Our main purpose of this paper is to give tight sufficient conditions for
Hamiltonicity and traceability of connected graphs and of connected balanced
bipartite graphs with given minimum degree, in terms of Wiener index and
Harary index, respectively. Furthermore, our work not only gives short
and unified proofs of previous work due to Hua and Wang \cite{HW2013},
and Yang \cite{Y2013}, but also generalizes all these theorems. Our main
tools come from Ning and Ge \cite{NG2015}, and Li and Ning \cite{LN2016},
respectively.

Recall that the \emph{Wiener index} of a connected graph $G$, denote by
$W(G)$, is defined to be the sum of distances between every pair of vertices
in $G$. That is,
$$
W(G)=\sum_{\{u,v\}\subseteq V(G)}d_G(u,v).
$$
The Harary index is also a useful topological index in chemical graph theory
and has received much attention during the past decades. This index has been
introduced in 1993 by Plav\v{s}i\'{c} et al. \cite{PNTM1993}
and by Ivanciue et al. \cite{IBB1993}, independently. For a connected graph $G$, the
\emph{Harary index} of $G$, denoted by $H(G)$, is defined as
$$
H(G)=\sum_{\{u,v\}\subseteq V(G)}\frac{1}{d_G(u,v)}.
$$
These two indices have found many applications in chemistry and there are
lots of papers dealing with these two indices, see surveys \cite{LL2013,XDT2015}.

In this paper, we mainly consider the following two problems related to
Wiener index, Harary index and Hamiltonian properties of graphs, which are
motivated by the main problems studied in \cite{LN2016}. The main results
in this paper are solutions to the following two problems.

\begin{prob}\label{Prob:1.1}
Among all non-Hamiltonian graphs (non-traceable graphs) $G$ of order
$n$ with $\delta(G)\geq k$, determine the values of $\min W(G)$
and $\max H(G)$, respectively.
\end{prob}

\begin{prob}\label{Prob:1.2}
Among all non-Hamiltonian balanced bipartite graphs $G$ of order
$2n$ with $\delta(G)\geq k$, determine the values of $\min W(G)$
and $\max H(G)$, respectively.
\end{prob}

We organize this paper as follows. In Section 2, we give some notes
on an old theorem about Hamilton cycles due to Erd\H{o}s and its
generalizations. As shown by Liu et al. \cite{LDJ2016,LDJ2017},
there are some errors in proofs of some previous work on traceability
of connected graphs, in terms of Wiener index and Harary index. We
remark that these theorems can be unified in a short proof (by using
the generalizations of Erd\H{o}s' theorem). In Section 3, we prove
the correct form and also prove a similar result on Hamilton cycles
in connected graphs. In Sections 4 and 5, by imposing the minimum
degree of graphs, we generalize the above results to connected graphs
and to connected balanced bipartite graphs, respectively. In the
last section, we give some ideas on traceability of connected
balanced and nearly-balanced bipartite graphs with given
minimum degree, still in terms of these two indices.

\section{Erd\H{o}s' theorem on Hamilton cycles, its refinements and some notes}
To find tight edge conditions for Hamilton cycles in graphs is a standard
topic in graph theory. In 1962, Erd\H{o}s \cite{E1962} proved the following
theorem, which generalized Ore's theorem \cite{O1961} by introducing the minimum
degree of a graph as a new parameter.

\begin{thm}[Main Theorem  in \cite{E1962}]
Let $G$ be a graph of order $n$. If $\delta(G)\geq k$, where $1\leq k\leq(n-1)/2$, and
$$
e(G)>\max\left\{\binom{n-k}{2}+k^2,\binom{\lceil(n+1)/2\rceil}{2}+\left\lfloor\frac{n-1}{2}\right\rfloor^2\right\},
$$
then $G$ is Hamiltonian.
\end{thm}

The original Erd\H{o}s' theorem has the following concise form, which is
listed as an exercise in West \cite{W1996}.

\begin{thm}[Exercise 7.2.28 in \cite{W1996}]
Let $G$ be a graph of order $n\geq 6k$ with $\delta(G)\geq k\geq 1$. If
$$e(G)>\binom{n-k}{2}+k^2$$ then $G$ is Hamiltonian.
\end{thm}

When $k=1$, a refinement of Erd\H{o}s' theorem can date back to Ore \cite{O1961},
and was also given by Bondy \cite{B1972}. (See also Exercise~28 on Page 126 of Bollob\'as' book 
\cite{B1978}). When $k=2$, Ning and Ge \cite{NG2015}
further proved the following refined theorem.

\begin{lem}[Lemma 2 in \cite{NG2015}]\label{Lem:2.3}
Let $G$ be a graph on $n\geq 5$ vertices and $m$ edges with $\delta\geq 2$.
If $m\geq\binom{n-2}{2}+4$, then $G$ is Hamiltonian unless $G\in
\mathcal{\mathcal{G}}_1=\{K_2\vee (K_{n-4}+2K_1),K_3\vee 4K_1,K_2\vee(K_{1,3}+K_1),
K_1\vee K_{2,4},K_3\vee (K_2+3K_1),K_4\vee 5K_1,K_3\vee(K_{1,4}+K_1),K_2\vee K_{2,5},
K_5\vee 6K_1\}$.
\end{lem}

As a corollary, Ning and Ge \cite{NG2015} also proved the following theorem
on traceability of connected graphs.

\begin{lem}[Lemma 4 in \cite{NG2015}]\label{Lem:2.4}
Let $G$ be a graph on $n\geq 4$ vertices and $m$ edges with $\delta \geq 1$.
If $m\geq \binom{n-2}{2}+2$, then $G$ is traceable unless $G\in \mathcal{G}_2=
\{K_1\vee(K_{n-3}+2K_1),K_2\vee4K_1,K_1\vee(K_{1,3}+K_1),K_{2,4},
K_2\vee(3K_1+K_2),K_3\vee 5K_1,K_2\vee(K_{1,4}+K_1),K_1\vee K_{2,5},
K_4\vee 6K_1\}$.
\end{lem}

Here, we would like to comment on some previous work on Wiener index, Harary
index and traceability of connected graphs. Hua and Wang \cite{HW2013} gave
a sufficient condition for traceability of connected graphs in terms of Harary
index. While in \cite{Y2013}, Yang gave a similar sufficient condition for
traceability of connected graphs in terms of Wiener index. However, as shown
by Liu et al., there are some errors in all the proofs. In two papers \cite{LDJ2016,LDJ2017},
Liu et al. have corrected the proof of Hua and Wang's result and Yang's result,
respectively. We point out that the proofs of Hua-Wang's result and Yang's result
can be unified by using Lemma~\ref{Lem:2.4} (together with some facts). Furthermore,
we will give a short and unified proof in the next section. All results in this
paper are given in the similar fashion.

\section{Corrected and unified forms of Hua-Wang's theorem and Yang's
 theorem}
In this section, we first prove a result on Hamiltonicity of connected graphs
with $\delta(G)\geq 2$, in terms of Wiener index and Harary index.

In order to state our results, we introduce some notation in \cite{LN2016}.
We define: for $1\leq k\leq(n-1)/2$, $L^k_n=K_1\vee(K_k+K_{n-k-1})$ and
$N^k_n=K_k\vee(K_{n-2k}+kK_1).$ Note that $L^1_{n}=N^1_{n}$. We denote by
$\bL_n^k$ and $\bN_n^k$ the graphs obtained from $L_{n+1}^{k+1}$ and
$N_{n+1}^{k+1}$, respectively, by deleting one vertex of degree $n$, i.e.,
for $0\leq k\leq n/2-1$,
$$\bL^k_n=K_{k+1}+K_{n-k-1} \mbox{ and } \bN^k_n=K_k\vee(K_{n-2k-1}+(k+1)K_1).$$

The next fact is useful. Since its proof is simple, we omit the proof.

\begin{fact}\label{Fact:3.1}
Let $G$ be a connected graph on $n$ vertices. Then there holds:
\begin{enumerate}[$(i)$]
\item $W(G)+e(G)\geq n(n-1)$, where the equality holds if and only if $diam(G)\leq 2$;

\item $e(G)\geq 2H(G)-\binom{n}{2}$, where the equality holds if and only if $diam(G)\leq 2$.
\end{enumerate}
\end{fact}

\begin{thm}\label{Thm:3.2}
Let $G$ be a connected graph of order $n\geq 5$, where $\delta(G)\geq 2$. If
$W(G)\leq W(N^2_n)$ or $H(G)\geq H(N^2_n)$, then $G$ is Hamiltonian unless
$G\in \mathcal{\mathcal{G}}_1$.
\end{thm}

\begin{proof} Since $diam(N^2_n)=2$, by Fact \ref{Fact:3.1}, we obtain that
$W(N^2_n)=n(n-1)-e(N^2_n)$ and $H(N^2_n)=\frac{1}{2}(e(N^2_n)+\binom{n}{2})$.

If $W(G)\leq W(N^2_n)$, then by Fact \ref{Fact:3.1}(i), we have $e(G)\geq n(n-1)-W(G)
\geq e(N^2_n)=\binom{n-2}{2}+4$. If $H(G)\geq H(N^2_n)$, then by Fact \ref{Fact:3.1}(ii),
we have $e(G)\geq 2H(G)-\binom{n}{2}\geq 2H(N^2_n)-\binom{n}{2}=e(N^2_n)=
\binom{n-2}{2}+4$.

By Lemma~\ref{Lem:2.3}, $G$ is Hamiltonian unless $G\in \mathcal{\mathcal{G}}_1$.
Furthermore, for every graph $G'\in \mathcal{\mathcal{G}}_1$, since $diam(G')=2$,
by Fact \ref{Fact:3.1}, we have $W(G')=n(n-1)-e(G')=n(n-1)-e(N^2_n)=W(N^2_n)$
and $H(G')=\frac{1}{2}(e(G')+\binom{n}{2})=\frac{1}{2}(e(N^2_n)+\binom{n}{2})=
H(N^2_n)$, where $n=|G'|$. This completes the proof.
\end{proof}

The second purpose of this section is to show that, some previous work
\cite{HW2013,Y2013} on Wiener index, Harary index, and the traceability
of connected graphs can be deduced directly from some structural lemma
due to the second author and Ge \cite{NG2015}. And these results can
be proved by a unified and short proof.

We firstly list some theorems due to Hua and Wang \cite{HW2013}, and due
to Yang \cite{Y2013}, respectively.

\begin{thm}[Theorem 2.2 in \cite{HW2013}]\label{Thm:3.3}
Let $G$ be a connected graph of order $n\geq 4$. If $H(G)\geq \frac{1}{2}
n^2-\frac{3}{2}n+\frac{5}{2}$, then $G$ is traceable, unless $G\in\{K_1\vee
(K_{n-3}+2K_1),K_2\vee (3K_1+K_2),K_4\vee 6K_1\}$.
\end{thm}

\begin{thm}[Theorem 2.2 in \cite{Y2013}]\label{Thm:3.4}
Let $G$ be a connected graph of order $n\geq 4$. If $W(G)\leq \frac{(n+5)(n-2)}{2}$,
then $G$ is traceable, unless $G\in \{K_1\vee(K_{n-3}+2K_1),K_2\vee(3K_1+K_2),
K_4\vee 6K_1\}$.
\end{thm}

Notice that $\underline{N}^1_n=K_1\vee (K_{n-3}+2K_1)$, $H(\underline{N}^1_n)
=\frac{1}{2}n^2-\frac{3}{2}n+\frac{5}{2}$, and $W(\underline{N}^1_n)=\frac{(n+5)(n-2)}{2}$.
In fact, the corrected forms of Theorems \ref{Thm:3.3} and \ref{Thm:3.4} include
six more extremal graphs, as shown by Liu et al. \cite{LDJ2016,LDJ2017}. In the
following, we write the clear form of Liu et al.'s theorems, and give a unified
and short proof, similar as the proof of Theorem \ref{Thm:3.2}.

\begin{thm}[Theorem 2.2 in \cite{LDJ2016} and Theorem 2.3 in \cite{LDJ2017}]
Let $G$ be a connected graph of order $n\geq 4$. If $W(G)\leq W(\underline{N}^1_n)$
or $H(G)\geq H(\underline{N}^1_n)$, then $G$ is traceable unless
$G\in \mathcal{\mathcal{G}}_2$.
\end{thm}
\begin{proof} Since $diam(\underline{N}^1_n)=2$, by Fact \ref{Fact:3.1}, we obtain
$W(\underline{N}^1_n)=n(n-1)-e(\underline{N}^1_n)$ and
$H(\underline{N}^1_n)=\frac{1}{2}(e(\underline{N}^1_n))+\binom{n}{2})$.

If $W(G)\leq W(\underline{N}^1_n)$, then by Fact \ref{Fact:3.1} (i), we have
$e(G)\geq n(n-1)-W(G)\geq n(n-1)-W(\underline{N}^1_n)=e(\underline{N}^1_n)=\binom{n-2}{2}+2$.
If $H(G)\geq H(\underline{N}^1_n)$, then by Fact \ref{Fact:3.1} (ii), we have
$e(G)\geq 2H(G)-\binom{n}{2}\geq 2H(\underline{N}^1_n)-\binom{n}{2}=e(\underline{N}^1_n)=\binom{n-2}{2}+2$.
By Lemma~\ref{Lem:2.4}, $G$ is traceable unless $G\in \mathcal{\mathcal{G}}_2$.

Furthermore, for every graph $G'\in \mathcal{\mathcal{G}}_1$, since $diam(G')=2$,
by Fact \ref{Fact:3.1}, we have $W(G')=n(n-1)-e(G')=n(n-1)-e(\underline{N}^1_n)=W(\underline{N}^1_n)$
and $H(G')=\frac{1}{2}(e(G')+\binom{n}{2})=\frac{1}{2}(e(\underline{N}^1_n)+\binom{n}{2})
=H(\underline{N}^1_n)$, where $n=|G'|$. This completes the proof.
\end{proof}

\section{Wiener index, Harary index and Hamiltonicity of connected graphs}
In this section, we will prove sharp results on traceability
and Hamiltonicity of connected graphs with given minimum degree,
in terms of Wiener index and Harary index. Our proofs depend on
a structural result due to Li and Ning \cite{LN2016}, which refines a theorem
of Erd\H{o}s \cite{E1962}.

To prove spectral analogs of Erd\H{o}s' theorem, Li and Ning
\cite{LN2016} proved the following refined form of the concise
Erd\H{o}s' theorem.

\begin{lem}[Lemma 2 in \cite{LN2016}]\label{Lem:4.1}
Let $G$ be a graph of order $n\geq 6k+5$, where $k\geq 1$. If
$\delta(G)\geq k$ and $$e(G)>\binom{n-k-1}{2}+(k+1)^2,$$ then
$G$ is Hamiltonian unless $G\subseteq L^k_n$ or $N^k_n$.
\end{lem}

\begin{lem}[Lemma 3 in \cite{LN2016}]\label{Lem:4.2}
Let $G$ be a graph of order $n\geq 6k+10$, where $k\geq 0$. If
$\delta(G)\geq k$ and
\[
e(G)>\binom{n-k-2}{2}+(k+1)(k+2),
\]
then $G$ is traceable unless $G\subseteq \bL_n^k$ or $\bN_n^k$.
\end{lem}

Next, we give solutions to Problem \ref{Prob:1.1} (when $n$ is sufficiently
large), whose proofs depend on the above structural lemmas.

\begin{thm}
Let $G$ be a connected graph of order $n\geq 6k+5$, where $\delta(G)\geq
k\geq 1$. If $W(G)\leq W(N^k_n)$ or $H(G)\geq H(N^k_n)$, then $G$ is
Hamiltonian unless $G=N^k_n$.
\end{thm}
\begin{proof} Since $diam(N^k_n)=diam(L^k_n)=2$, by Fact \ref{Fact:3.1}, we obtain
that $W(G')=n(n-1)-e(G')$ and $H(G')=\frac{1}{2}(e(G')+\binom{n}{2})$,
if $G'\in \{N^k_n,L^k_n\}$.

If $W(G)\leq W(N^k_n)$, then by Fact \ref{Fact:3.1} (i), we have $e(G)\geq
n(n-1)-W(N^k_n)\geq n(n-1)-(n(n-1)-e(N^k_n))=e(N^k_n)=\binom{n-k}{2}+k^2>
\binom{n-k-1}{2}+(k+1)^2$ when $n>3k+2$. If $H(G)\geq H(N^k_n)$, then by
Fact \ref{Fact:3.1} (ii), we also have $e(G)\geq 2H(N^k_n)-\binom{n}{2}=
e(N^k_n)=\binom{n-k}{2}+k^2>\binom{n-k-1}{2}+(k+1)^2$ when $n>3k+2$. By
Lemma~\ref{Lem:4.1}, $G$ is Hamiltonian unless $G\subseteq L^k_n$ or $N^k_n$.

If $G\subsetneqq N^k_n$, then $W(G)>W(N^k_n)$ and $H(G)<H(N^k_n)$, a contradiction.
Recall that $e(N^k_n)=\binom{n-k}{2}+k^2$ and $e(L^k_n)=\binom{n-k}{2}+\frac{(k+1)k}{2}$.
Thus $e(N^k_n)>e(L^k_n)$ when $k\geq 2$ and $e(N^k_n)=e(L^k_n)$ when $k=1$.
Hence $W(L^k_n)>W(N^k_n)$ when $k\geq 2$ and $W(L^1_n)=W(N^1_n)$; $H(L^k_n)
<H(N^k_n)$ when $k\geq 2$ and $H(L^1_n)=H(N^1_n)$. So, if $G\subseteq L^k_n$
and $k\geq 2$, then $W(G)\geq W(L^k_n)>W(N^k_n)$ and $H(G)\leq H(L^k_n)<H(N^k_n)$,
a contradiction. It follows $G=N^k_n$ when $k\geq 2$ or $k=1$ (in this case,
$G=L^1_n=N^1_n$). This completes the proof.
\end{proof}

\begin{thm}
Let $G$ be a connected graph of order $n\geq 6k+10$, where $\delta(G)\geq k\geq 1$.
If $W(G)\leq W(\underline{N}^k_n)$ or $H(G)\geq H(\underline{N}^k_n)$, then $G$ is
traceable unless $G=\underline{N}^k_n$.
\end{thm}

\begin{proof}
Since $diam(\underline{N}^k_n)=2$, by Fact \ref{Fact:3.1}, we obtain that $W(G')=n(n-1)-e(G')$
and $H(G')=\frac{1}{2}(e(G')+\binom{n}{2})$, if $G'=\underline{N}^k_n$.

If $W(G)\leq W(\underline{N}^k_n)$, by Fact \ref{Fact:3.1} (i), we have $e(G)\geq n(n-1)-
W(\underline{N}^k_n)=n(n-1)-(n(n-1)-e(\underline{N}^k_n))=e(\underline{N}^k_n)
=\binom{n-k-1}{2}+k(k+1)>\binom{n-k-2}{2}+(k+1)(k+2)$ when $n>3k+4$. If $H(G)
\geq H(\underline{N}^k_n)$, by Fact 1 (ii), we also have $e(G)\geq 2H(\underline{N}^k_n)
-\binom{n}{2}=e(\underline{N}^k_n)=\binom{n-k-1}{2}+k(k+1)>\binom{n-k-2}{2}+(k+1)(k+2)$
when $n>3k+4$. By Lemma~\ref{Lem:4.2}, $G$ is traceable unless $G\subseteq \underline{L}^k_n$
or $\underline{N}^k_n$. Since $G$ is connected, we have $G\subseteq\underline{N}^k_n$.

If $G\varsubsetneq \underline{N}^k_n$, then $W(G)>W(\underline{N}^k_n)$ and $H(G)<H(\underline{N}^k_n)$,
a contradiction. Thus, $G=N^k_n$. This completes the proof.
\end{proof}

\section{Wiener index, Harary index and Hamiltonicty of connected balanced bipartite graphs}
In this section, we will prove sharp results on Hamiltonicity of connected balanced
bipartite graphs with given minimum degree, in terms of  Wiener index and Harary index.
Our proofs depend on the following structural result due to Li and Ning, which refines
a theorem of Moon and Moser \cite{MM1963}.

\begin{lem}[Lemma 5 in \cite{LN2016}]\label{Lem:5.1}
Let $G$ be a balanced bipartite graph of order $2n$. If
$\delta(G)\geq k\geq 1$, $n\geq 2k+1$ and
$$e(G)>n(n-k-1)+(k+1)^2,$$ then $G$ is Hamiltonian unless $G\subseteq B_n^k$.
\end{lem}

In the above theorem, we define $B^k_n$ ($1\leq k\leq n/2$) as the graph obtained from
$K_{n,n}$ by deleting all edges in its one subgraph $K_{n-k,k}$. Note that
$e(B_n^k)=n(n-k)+k^2$ and $B_n^k$ is not Hamiltonian.

The following useful fact is simple, and we omit the proof.

\begin{fact}\label{Fact:5.1}
Let $G$ be a connected balanced bipartite graph of order $2n$. Then there holds:
\begin{enumerate}[$(i)$]
\item $e(G)+3(n^2-e(G))+4\binom{n}{2}\leq W(G)$, where the equality holds if and
only if for any two vertices $x,y$, if $x,y$ are in different partition sets then
$d(x,y)\leq 3$, and if $x,y$ are in the same partition set then $d(x,y)=2$;

\item $e(G)+\frac{1}{3}(n^2-e(G))+\binom{n}{2}\geq H(G)$, where the equality holds
if and only if for any two vertices $x,y$, if $x,y$ are in different partition sets
then $d(x,y)\leq 3$, and if $x,y$ are in the same partition set then $d(x,y)=2$.
\end{enumerate}
\end{fact}

The following theorem gives a solution to Problem \ref{Prob:1.2}.

\begin{thm}\label{5.1}
Let $G$ be a connected balanced bipartite graph of order $2n$, where $n\geq 2k+2$
and $\delta(G)\geq k\geq 1$. If $W(G)\leq W(B^k_n)$ or $H(G)\geq H(B^k_n)$, then
$G$ is Hamiltonian unless $G=B_n^k$.
\end{thm}

\begin{proof}
By Fact \ref{Fact:5.1}, we have $W(B^k_n)=5n^2-2n-2e(B^k_n)$ and $H(B^k_n)=
e(B^k_n)+3(n^2-e(B^k_n))+\binom{n}{2}$. If $W(G)\leq W(B^k_n)$, then
$e(G)\geq\frac{1}{2}(5n^2-2n-W(G))\geq\frac{1}{2}(5n^2-2n-W(B^k_n))=e(B^k_n)$.
If $H(G)\geq H(B^k_n)$, then $e(G)\geq e(B^k_n)$. When $n\geq 2k+2$,
$e(B^k_n)=n(n-k)+k^2>n(n-k-1)+(k+1)^2$. By Lemma~\ref{Lem:5.1}, $G$ is
Hamiltonian unless $G\subseteq B_n^k$. If $G\varsubsetneq B_n^k$, then
$W(G)>W(B_n^k)$, a contradiction. This completes the proof.
\end{proof}

Since the bound in Theorem~\ref{5.1} is tight, some previous work (see Theorem 2.2
in \cite{Z2013}) in this direction is a direct corollary.

\section{Concluding remarks}
One may ask to study the traceability of connected bipartite graphs with given
minimum degree. We know such a graph should be balanced or nearly-balanced.
Recently, Li and Ning \cite{LN2016-arxiv} studied spectral conditions for
traceability of bipartite graphs with given minimum degree. The study of
traceability of bipartite graphs in terms of Wiener index and Harary index
is very similar to the ones in \cite{LN2016-arxiv}. Some structural theorems
developed in \cite{LN2016-arxiv} about traceability of connected balanced
and nearly-balanced bipartite graphs will play the central roles in proofs
of these results. We omit the details and refer them to the interested reader.

\section{Acknowledgment}
H.-B. Hua was supported by the National Natural Science Foundation of China, No.\ 11571135.
B. Ning was supported by the National Natural Science Foundation of China, No.\ 11601379.

After finishing the original version of this paper, the second author has
been told that Dr. Liu et al. have proved the corrected forms of Hua-Wang's
theorem and Yang's theorem independently, when he was visiting East China
University of Science and Technology on Aug. 19th-25th, 2016. The second
author thanks to Dr. Liu for giving him the preprints of \cite{LDJ2016,LDJ2017}.

\end{document}